\documentclass[12pt,leqno]{amsart}

\usepackage{amssymb}
\usepackage{amsmath,amssymb,color}
\usepackage{cite}
\def\BC{\mathbb C}
\def\BN{\mathbb N}
\def\BR{\mathbb R}
\def\cD{\mathcal D}
\def\rIm{\mathrm{Im}}
\def\rRe{\mathrm{Re}}
\def\rd{\mathrm d}
\def\e{\mathrm e}
\def\ri{\mathrm i}
\def\Ga{\Gamma}
\def\Om{\Omega}
\def\al{\alpha}
\def\be{\beta}
\def\ga{\gamma}
\def\de{\delta}

\def\ve{\varepsilon}
\def\te{\theta}
\def\la{\lambda}
\def\si{\sigma}
\def\vp{\varphi}
\def\f{\frac}
\def\ov{\overline}
\def\pa{\partial}
\def\wh{\widehat}
\def\wt{\widetilde}

\def\BCC{_{0}C^1}
\def\HH{H_\al}
\def\paa{\pa_t^\al}
\def\ddda{\rd_t^\al}
\theoremstyle{thmstyle}
\newtheorem{thm}{Theorem}[section]

\newtheorem{lem}[thm]{Lemma}
\newtheorem{coro}[thm]{Corollary}
\newtheorem*{prob}{Problem}
\newtheorem{rem}[thm]{Remark}

\numberwithin{equation}{section}
\allowdisplaybreaks[4]

\title[Unique continuation and inverse source problem]
{Inverse source problem for a one-dimensional
time-fractional diffusion equation and
unique continuation for weak solutions}

\author[Z. Li]{Zhiyuan Li}
\address{Zhiyuan Li\newline
School of Mathematics and Statistics, Ningbo University, 818 Fenghua Road,
Ningbo, Zhejiang 315211, China.}
\email{lizhiyuan@nbu.edu.cn}

\author[Y. Liu]{Yikan Liu}
\address{Yikan Liu\newline
Research Center of Mathematics for Social Creativity, Research Institute for
Electronic Science, Hokkaido University, N12W7, Kita-Ward, Sapporo
060-0812, Japan.}
\email{ykliu@es.hokudai.ac.jp}

\author[M. Yamamoto]{Masahiro Yamamoto}
\address{Masahiro Yamamoto\newline
Graduate School of Mathematical Sciences, The University of Tokyo,
3-8-1 Komaba, Meguro-ku, Tokyo 153-8914, Japan;
Honorary Member of Academy of Romanian Scientists,
Ilfov, nr. 3, Bucuresti, Romania;
Correspondence member of Accademia Peloritana dei Pericolanti,
Palazzo Universit\`a, Piazza S. Pugliatti 1 98122 Messina, Italy.} 
\email{myama@next.odn.ne.jp}


\date{}
\begin{document}
\maketitle

\begin{abstract}
In this paper, we obtain the sharp uniqueness for an inverse $x$-source
problem for a one-dimensional time-fractional diffusion equation with a
zeroth-order term by the minimum possible lateral Cauchy data. The key
ingredient is the unique continuation which holds for weak solutions.
\end{abstract}

%%%%%%%%%%%%%%%%%%%%%%%%%%%%%%%%%%%%%%%%

\section{Introduction and main results}

As a representative of various nonlocal models, time-fractional diffusion
equations have attracted consistent attention of multidisciplinary researchers
owing to their capability in describing anomalous diffusion (e.g.,
\cite{AG92,HH98}). In the last decades, fundamental theory has been developed
rapidly for time-fractional diffusion equations, represented by the fundamental
solution and the well-posedness results established e.g.\! in
\cite{EK04,KRY21}.
Remarkably, it reveals in \cite{LY19,SY11} that time-fractional diffusion
equations inherit the time analyticity and the maximum principle from their
integer counterparts. On the contrary, they differ considerably in view of the
long-time asymptotic behavior (see \cite{SY11}).

In contrast to the above results, some issues remain open in the uniqueness of
the lateral Cauchy problem and the unique continuation for time-fractional
diffusion equations. As is known, these two properties essentially characterize
parabolic equations in the sense of the infinite propagation speed of local
information of homogeneous equations (see \cite{SS87}). However, for
time-fractional diffusion equations, most literature only obtained weak unique
continuation because additional assumptions were required on the boundary or
at the initial time (see \cite{CLN13,HLY,JLLY,LN16,LN19,SY11,XCY11}).

In this article, we are concerned with an inverse source problem
for a one-dimensional time-fractional diffusion equation with a potential
term, which is described tentatively:
\begin{equation}\label{eq-1.1}
\ddda y(x,t) - y_{x x}(x,t) +p(x)y=\rho(t)f(x)\quad\mbox
{in }(0,1)\times(0,T).
\end{equation}
Here the Caputo derivative $\ddda$ is defined by
$$
\ddda w(t) = \f1{\Ga(1-\al)}\int^t_0(t-\tau)^{-\al}\f{\rd w}{\rd\tau}(\tau)\,
\rd\tau
\quad\mbox{for $w\in C^1[0,T]$ or $\in W^{1,1}(0,T)$,}
$$
where $0<\al<1$ and $\Ga(\,\cdot\,)$ is the gamma function.

Formulation \eqref{eq-1.1} is quite conventional, but the definition
of $\ddda$ requires the differentiability of $y$ in $t$.
On the other hand, if for $y$,
we require the $H^2(0,1)$-regularity in $x$ and
$L^2(0,T)$-regularity in $t$, then $\ddda y$ should be in
$L^2$-space in $x$ and $t$, which is interpreted to be
weaker regularity than the one-time differentiability in time.
Moreover, for $\rho f\in L^2((0,1) \times (0,T))$, it is required that
$\ddda y \in L^2((0,1)\times (0,T))$. Thus, for such a class of
solutions, we are suggested to exploit the class of
functions $w=w(t)$ such that $\ddda w \in L^2(0,T)$, and the
class $W^{1,1}(0,T)$ apparently seems narrow for this requirement.
Therefore, we start to reformulate the time-fractional derivative
in adequate Sobolev spaces.  We emphasize that such formulated
time-fractional derivatives allow us to work within the regularity
specified by \eqref{eq-1.3} below.

First we set
$$
\BCC[0,T] := \{ w \in C^1[0,T];\, w(0) = 0\}.
$$
We consider the Caputo derivative $\ddda w(t)$ ($0<\al<1$)
as an operator from $\BCC[0,T]$ to $L^2(0,T)$.
In other words, by setting $\cD(\ddda) = {\BCC}[0,T]$, we define
the domain $\cD(\ddda)$.
The operator $\ddda$ with domain $\BCC[0,T]$ is not a closed operator, but it admits a
unique minimum closed extension, which is
denoted by $\paa$ (e.g., Kubica, Ryszewska and Yamamoto \cite{KRY21}).
We can characterize the domain $\cD(\paa)$ as follows.
We recall the Sobolev-Slobodecki space
$H^\al(0,T)$ with the norm $\|\cdot\|_{H^\al(0,T)}$ is defined by
$$
\| w\|_{H^\al(0,T)}:=
\left( \| w\|^2_{L^2(0,T)}
+ \int^T_0\!\!\!\int^T_0\f{|w(t)-w(\tau)|^2}{|t-\tau|^{1+2\al}}
\,\rd t \rd\tau \right)^{\f12}
$$
(see e.g.\! Adams \cite{Ad}).
We define the Riemann-Liouville integral operator $J^\be$ for $\be>0$ as
\[
J^\be w(t):=\f1{\Ga(\be)}\int_0^t(t-\tau)^{\be-1}w(\tau)\,\rd\tau,
\quad w\in L^2_{\mathrm{loc}}(0, +\infty).
\]

Setting
$$
\HH(0,T):= \ov{\BCC[0,T]}^{H^\al(0,T)},
$$
we see that
$$
\HH(0,T) =
\left\{\!\begin{alignedat}{2}
& H^\al(0,T), & \quad & 0<\al<\f12, \\
& \left\{w\in H^{\f12}(0,T);\,\int^T_0\f{|w(t)|^2}t\,\rd t<\infty\right\},
& \quad & \al=\f12,\\
& \{w\in H^\al(0,T);\,w(0)=0\}, & \quad & \f12 < \al < 1
\end{alignedat}\right.
$$
and
$$
\| w\|_{\HH(0,T)} =
\left\{\! \begin{alignedat}{2}
& \|w\|_{H^\al(0,T)}, & \quad & \al \ne \f12, \\
& \left(\|w\|_{H^{\f12}(0,T)}^2 + \int^T_0 \f{|w(t)|^2}t\,\rd t\right)^{\f12},
& \quad & \al=\f12.
\end{alignedat}\right.
$$
Then it is known that (see e.g. \cite{GLY,KRY21})
$$
\HH(0,T) = J^\al L^2(0,T),\quad 0<\al<1.
$$
The minimum closed extension $\paa$ of the operator $\ddda$
with the domain $\cD(\ddda) =\, {\BCC[0,T]}$ satisfies
$$
\paa = (J^\al)^{-1}, \quad \cD(\paa) =\HH(0,T)
$$
and there exists
a constant depending only on $\al$ such that
$$
C^{-1}\| w\|_{\HH(0,T)} \le \| \paa w\|_{L^2(0,T)}
\le C\| w\|_{\HH(0,T)} \quad \mbox{for all $w\in\HH(0,T)$}
$$
(see Gorenflo, Luchko and Yamamoto \cite{GLY},
Kubica, Ryszewska and Yamamoto \cite{KRY21}, Yamamoto \cite{Ya22}).

Throughout this article, instead of the time-fractional
diffusion equation \eqref{eq-1.1}, we consider
\begin{equation}\label{eq-1.2}
\paa y(x,t) - y_{x x}(x,t) + p(x)y(x,t) = \rho(t)f(x) \quad
\mbox{in }(0,1) \times (0,T).
\end{equation}
Here $\rho$ and $f$ stand for the temporal and spatial components of the
source term, respectively.

In the sequel, we set $F(x,t) = \rho(t)f(x)$ or
$F\equiv 0$, and
$$
a\in L^2(0,1), \quad F \in L^2(0,T;L^2(0,1)),\quad
p\in L^\infty(0,1).
$$
In general, we define a solution to
a time-fractional diffusion equation with initial value
$a \in L^2(0,1)$ as follows:
\begin{equation}\label{1.3}
\paa (u-a)(x,t) - u_{x x}(x,t) + p(x)u(x,t) = F(x,t) \quad
\mbox{in $L^2(0,T;H^{-1}(0,1))$}
\end{equation}
and
\begin{equation}\label{1.4}
u-a \in\HH(0,T;H^{-1}(0,1)), \quad
u \in L^2(0,T;H^1(0,1)),
\end{equation}
where $H^{-1}(0,1) = (H^1_0(0,1))'$ is the dual space of $H^1_0(0,1)$.
Here we remark that
$u \in L^2(0,T;H^1(0,1))$ implies $u_{x x}
\in L^2(0,T;H^{-1}(0,1))$.  Indeed, for almost all $t \in (0,T)$,
since $u_x(\,\cdot\,,t)\in L^2(0,1)$, we see
$$
{}_{H^{-1}(0,1)}\langle u_{x x}(\,\cdot\,,t), \phi\rangle_{H^1_0(0,1)}
= -(u_x,\phi_x)_{L^2(0,1)} \quad \mbox{for all
$\phi \in H^1_0(0,1)$},
$$
and
$$
|{}_{H^{-1}(0,1)}\langle u_{x x}(\,\cdot\,,t), \phi\rangle_{H^1_0(0,1)}|
\le \| u_x(\,\cdot\,,t)\|_{L^2(0,1)}\| \phi\|_{H^1_0(0,1)}.
$$
Therefore, $u_{x x}(\,\cdot\,,t)$, whose derivative is
taken in the sense of distribution,
can define a bounded linear functional on $H^1_0(0,1)$, that is,
$u_{x x}(\,\cdot\,,t) \in H^{-1}(0,1)$.

This class defined by \eqref{1.4} is compatible with the function space for the
initial-boundary value problem. For example, attaching \eqref{1.3}
with the boundary condition $u(0,t) = u(1,t) = 0$ for $0<t<T$,
if $F \in L^2(0,T;H^{-1}(0,1))$,
then we can prove the unique existence of solution to the initial
boundary value problem within the above class
(e.g., \cite{KRY21}).

Since $\HH(0,T)$ is the closure of the set $\BCC[0,T]$ of
$C^1$-functions vanishing at $t=0$ by the norm of
$H^\al(0,T)$, we can interpret the first regularity condition in \eqref{1.4}
as generalized initial condition.
In particular, for $\f12 < \al < 1$,
in view of the Sobolev embedding
$\HH(0,T) \subset H^\al(0,T)\subset C[0,T]$,
if $u-a \in\HH(0,T;H^{-1}(0,1))$, then $u-a \in C([0,T];H^{-1}(0,1))$,
and so $u$ satisfies the initial condition
$u(\,\cdot\,,0) = a$ in $H^{-1}(0,1)$.

Throughout this article, we assume that a solution $y$ to \eqref{eq-1.2} satisfies
\begin{equation}\label{eq-1.3}
y \in \HH(0,T;H^1(0,1)).
\end{equation}
We recall that we consider the zero initial condition in the sense
of $y \in\HH(0,T;H^1(0,1))$.

In addition to the regularity \eqref{1.4}, for any non-empty open interval
$I$ such that $\ov I \subset (0,1)$, we can prove
\begin{equation}\label{1.6}
J^\al u \in L^2(0,T;H^2(I)).
\end{equation}
In particular, the trace theorem yields
$$
(J^\al u)_x(x_0,\cdot\,) = (J^\al u_x)(x_0,\cdot\,)
\in L^2(0,T)
$$
for arbitrary $x_0 \in (0,1)$.  For completeness, we  provide the proof of
\eqref{1.6} in Appendix \ref{sec-app}.

For \eqref{eq-1.2}, our target of this article is the uniqueness for the
following inverse source problem:

\begin{prob}
Fix constants $T>0$ and $x_0\in (0,1)$. Let $y\in \HH(0,T;H^1(0,1))$ satisfy
\eqref{eq-1.2} with $p\in L^\infty(0,1)$.
Can we uniquely determine $f\in L^2(0,1)$ by data
$y(x_0,\cdot\,)$ and $(J^\al y)_x(x_0,\cdot\,)$ in $(0,T),$
provided that $\rho$ is given suitably$?$
\end{prob}

By \eqref{1.6} and the trace theorem, we note that the data
$(J^\al y)_x(x_0,\cdot\,)$ can make sense as function in $L^2(0,T)$.
For a sufficiently smooth initial value $a$, one can observe
$y(x_0,t)$ and $y_x(x_0,t)$ in real applications, which means the data
of the concentration and its rate of change at a single point. For
$a\in L^2(0,1)$, $y_x(x_0,t)$ does not make sense in $L^2(0,T)$, and
more practical observation data can be taken in $I\times(0,T)$ with small
open interval $I$ including $x_0$.

In the above problem, by the term $\paa y$, we emphasize that we can consider the zero initial value.
On the other hand, boundary values are unknown in
the above problem, and we just treat any
function $y \in \HH(0,T;H^1(0,1))$
satisfying \eqref{eq-1.2}. Therefore, the above problem
requires the unique determination of $f(x)$ without data
on the lateral boundary $\{0,1\} \times (0,T)$.
Our first main result is concerned with such sharp uniqueness:

\begin{thm}\label{thm-isp}
Fix constants $T>0$ and $x_0\in (0,1)$ arbitrarily.
We assume that $y\in \HH(0,T;H^1(0,1))$
satisfies \eqref{eq-1.2} with $p\in L^\infty(0,1),$ $f\in L^2(0,1)$
and $\rho\in H^1(0,T)$ satisfying $\rho(0)\ne0$. Then
$J^\al y(x_0,\cdot\,)=(J^\al y)_x(x_0,\cdot\,)=0$
in $(0,T)$ implies $f\equiv0$ in $(0,1)$.
\end{thm}

Since $y(x_0,\cdot\,) \in L^2(0,T)$, we see that
$y(x_0,\cdot\,)=0$ in $(0,T)$ if and only if
$J^\al y(x_0,\cdot\,)=0$ in $(0,T)$.

With arbitrarily chosen point $x_0 \in (0,1)$,
only two $t$-dependent functions $y(x_0,\cdot\,)$ and
$(J^\al y)_x(x_0,\cdot\,)$ are available for determining $f$.
In particular, we do not need the boundary values,
which are required in most literature.
This turns out to be novel compared with all existing results on inverse
problems for time-fractional diffusion equations (see \cite{LLY19} and
the references therein).

On the other hand, let us consider
$$
y(x,t)=\f1{4\pi^2}(E_{\al,1}(-4\pi^2t^\al)-1)\sin2\pi x,\quad0<x<1,\ t>0,
$$
where $E_{\al,\be}(z):=\sum_{k=0}^\infty\f{z^k}{\Ga(\al k+\be)}$ ($\be>0$)
is the Mittag-Leffler function defined for $z\in\BC$. It is known that
$E_{\al,1}(z)$ is an entire function in $z\in\BC$ and
$\pa_t^\al(E_{\al,1}(-4\pi^2t^\al)-1)=-4\pi^2E_{\al,1}(-4\pi^2t^\al)$ for $t>0$
(e.g., \cite{P99}). Then we can directly verify that
$$
\pa_t^\al y(x,t)-y_{x x}(x,t)=\rho_0(t)f_0(x),\quad0<x<1,\ t>0,
$$
where $\rho_0(t)=-1$ and $f_0(x)=\sin2\pi x$, and that $y(\f12,t)=0$ for $t>0$
in spite of $f_0\not\equiv0$. This example indicates that only data
$y(x_0,\cdot\,)$ does not yield the uniqueness even though we have the zero
boundary data $y(0,t)=y(1,t)=0$ for $t>0$. In this sense, in Theorem \ref{thm-isp},
data $y(x_0,\cdot\,)$ and $y_x(x_0,\cdot\,)$ can be considered as the minimum
possible for the uniqueness.

In order to prove Theorem \ref{thm-isp}, we need the
uniqueness for the lateral Cauchy problem for the homogeneous equation
with non-zero initial value $a$:
\begin{equation}\label{eq-1.4}
\pa_t^\al (u(x,t)- a(x)) -u_{x x}(x,t) + p(x)u(x,t)=0
\quad\mbox{in }(0,1)\times(0,T).
\end{equation}
Thus we can state our second main result:

\begin{thm}\label{thm-Cauchy}
Fix constants $T>0$ and $x_0 \in (0,1)$, $p\in L^{\infty}(0,1)$
arbitrarily.
Let $u\in L^2(0,T;H^1(0,1))$ satisfy $u-a \in\HH(0,T;H^{-1}(0,1))$ and
\eqref{eq-1.4} with some $a\in L^2(0,1)$. Then
$J^\al u(x_0,\cdot\,) = J^\al u_x(x_0,\cdot\,)=0$ in $(0,T)$
implies $u = 0$ in $(0,1)\times(0,T)$ and $a=0$ in $(0,1)$.
\end{thm}

Here by \eqref{1.6}, we note that $J^\al u \in L^2(0,T;H^2(I))$ with
open interval $I$ such that $\ov I \subset (0,1)$ and so again
$J^\al u_x(x_0,\cdot\,) \in L^2(0,T)$.

In Theorem \ref{thm-Cauchy}, we notice that not only boundary values but also
an initial value $a$ are unknown, and the theorem concludes $a=0$ in $(0,1)$
as well as  $u=0$ in $(0,1) \times (0,T)$.
The same result holds for one-dimensional parabolic equations,
which is an immediate corollary of the well-known uniqueness of the lateral
Cauchy problem (see e.g.\! Isakov \cite{I06}) by treating $(0,x_0)\times(0,T)$ and
$(x_0,1)\times(0,T)$ separately.

The observation point $x_0$ in Theorems \ref{thm-isp}--\ref{thm-Cauchy} is
restricted to the open interval $(0,1)$ because the proof of Theorem \ref{thm-Cauchy}
relies on the interior regularity theory in Gilbarg and Trudinger \cite{GT}. The
generalization to allowing a boundary point $x_0\in\{0,1\}$ seems not trivial and
we will not discuss this issue in this article.

The following corollary is an immediate consequence of Theorem
\ref{thm-Cauchy}.

\begin{coro}[classical unique continuation]\label{thm-ucp}
Choose a constant $T>0$ and a nonempty open interval $I\subset(0,1)$
arbitrarily.
Let $u\in L^2(0,T; H^1(0,1))$ satisfy $u-a \in\HH(0,T;$ $H^{-1}(0,1))$
and \eqref{eq-1.4} with $p\in L^\infty(0,1)$. Then
$u=0$ in $I\times(0,T)$ implies $u\equiv0$ in $(0,1)\times(0,T)$.
\end{coro}

In the case of $p\equiv 0$, the uniqueness as in Theorem \ref{thm-Cauchy} and
Corollary \ref{thm-ucp} was proved in Li and Yamamoto \cite{LY-FCAA}.
The proof in \cite{LY-FCAA} is not applicable directly to our case
$p \not\equiv 0$, and moreover \cite{LY-FCAA} requires the higher
regularity for the solution $u$.

As important contribution to the unique continuation for time-fractional
partial differential equations, we refer to
Lin and Nakamura \cite{LN21}, which proves the uniqueness
for more general time-fractional partial differential equations with order
$\al \in (0,1) \cup (1,2)$ in a bounded domain $\Om \subset
\BR^d$.
In \cite{LN21}, the unique continuation was proved under assumptions
that the coefficients of the equations are in $C^\infty$-class and
solutions $u$ are strong solutions, that is, satisfy
$$
u \in L^2(0,T;H^2(\Om)) \cap H^\al(0,T;L^2(\Om)),
$$
while we establish the unique continuation within the weaker regularity
\eqref{1.4}.
Moreover, for $0<\al<\f12$, the space $H^\al(0,T;L^2(\Om))$ is not
in $C([0,T];L^2(\Om))$ and so the initial condition requires special cares
even if initial values are known. In particular, the regularity
$H^\al(0,T;L^2(\Om))$ itself does not justify any initial conditions for $\al<\f12$.
As for the formulation of initial condition, we remark that in \cite{LN21},
the fractional derivative $\paa$ is defined through some
extension in $t$ of $u$ from $0<t<T$ to $-\infty<t<\infty$,
and their formulation is not the same as ours.
Thus we should understand that our result and the one in \cite{LN21} are
independent, as long as we are limited to the one-dimensional case.

The remainder of this article is organized as follows. In Section
\ref{sec-prelim},
we collect some necessary ingredients for proving the main results. Then
Sections \ref{sec-proof} and \ref{sec-isp} are devoted to the proofs of
Theorems \ref{thm-Cauchy} and \ref{thm-isp}, respectively.
Concluding remarks will be provided in Section \ref{sec-rem},
and Appendix \ref{sec-app} is devoted to the verifications of
some technical details.
%%%%%%%%%%%%%%%%%%%%%%%%%%%%%%%%%%%%%%%%

\section{Preliminaries}\label{sec-prelim}

First we fix some frequently used notations. We set $\BR_+:=(0,+\infty)$ and
denote the Laplace transform of $w\in L_{\mathrm{loc}}^1(\BR_+)$ by
\[
\wh w(s):=\int_{\BR_+}w(t)\,\e^{-s t}\,\rd t,\quad s>s_0,
\]
provided that the integral converges for some constant $s_0 > 0$.
Henceforth we write
$\vp'(x) = \f{\rd\vp}{\rd x}(x)$, etc.\! if there is no fear of confusion.

Let $\de > 0$ be arbitrarily fixed.
For $z\in\BC$, by $\vp(x,z)$ we denote the solution to the following initial
value problem for a second order ordinary differential equation
\begin{equation}\label{eq-y}
\begin{cases}
-\vp''+p(x)\vp=z^2\vp, & x>\de,\\
\vp(\de)=1,\ \vp'(\de)=0.
\end{cases}
\end{equation}
It is known (e.g., \cite[Theorem 1.1]{lev})
that the solution $\vp$ is analytic with respect to the parameter $z\in\BC$.
Moreover, we also have the following asymptotic formulae for $\vp(x,z)$ as
$|z|\to\infty$.

\begin{lem}\label{lem-asymp}
The solution $\vp(x,z)$ to \eqref{eq-y} admits the asymptotic estimate
\begin{equation}\label{esti-asymp0}
\vp(x,z)=O(\e^{|\rIm\,z|x})\quad(|z|\to\infty),
\end{equation}
and more precisely
\begin{equation}\label{esti-asymp}
\vp(x,z)=\cos(z (x-\de))+O\left(\f{\e^{|\rIm\,z|x}}{|z|}\right)
\quad \mbox{as $|z|\to\infty$}.
\end{equation}
Moreover$,$ if $\arg\,z=\pm\f\pi2,$ then there exists $\eta_0>0$
such that for $|z|>\eta_0,$ there holds
\begin{equation}\label{esti-phi'}
|\vp'(x,z)|\le C|z|\,\e^{|z|},\quad x>\de.
\end{equation}
\end{lem}

\begin{proof}
The first two asymptotic estimates for $\vp(x,z)$ can be found in
\cite[Chapter 1, Lemma 2.1]{lev}. Then it remains to show the inequality
\eqref{esti-phi'}. Indeed, for $z=\ri\,\eta$ with $\eta\in\BR$, we integrate
both
sides of the governing equation in \eqref{eq-y} from $\de$ to $x$ and employ the
initial condition in \eqref{eq-y} to deduce
\[
\vp'(x,z)=\int_\de^x(p(\xi)+\eta^2)\vp(\xi,z)\,\rd\xi.
\]
Substituting the asymptotic estimate \eqref{esti-asymp0} for $\vp(x,z)$ into
the above equation yields
\[
|\vp'(x,z)|\le C\left(\|p\|_{L^\infty(0,1)}+\eta^2\right)
\int_0^x\e^{|\eta|\xi}\,\rd\xi=\f{C(\|p\|_{L^\infty(0,1)}+\eta^2)}{|\eta|}
\,\e^{|\eta|}\le C|\eta|\,\e^{|\eta|}
\]
for $|\eta|>\eta_0$, which completes the proof.
\end{proof}
%%%%%%%%%%%%%%%%%%%%

Next, we recall two useful results from the complex analysis.

\begin{lem}[generalized Liouville's theorem]\label{lem-GL}
Assume that $f$ is an entire function and
there exist constants $N\in \BN$ and $R>0$ such that
$|f(z)|\le C|z|^N$ for $|z|>R$. Then $f$ is a polynomial of order
at most $N$.
\end{lem}

\begin{lem}[Phragm\'en-Lindel\"of principle]\label{lem-PL}
Fix constants $\te_2>\te_1$ and let $F$ be a holomorphic function in a
sector $S:=\{z\in\BC; \, \te_1<\arg z<\te_2\}$. Assume that $F$ is continuous
on the closure of $S$ and $|F|\le1$ on the boundary of $S$. If there exist
constants $\ga\in[0,\f\pi{\te_2-\te_1})$ and $C>0$ such that
\[
|F(z)|\le C\exp(C|z|^\ga),\quad\forall\,z\in S,
\]
then $|F|\le1$ in $S$.
\end{lem}

Lemma \ref{lem-GL} can be proved by using the series theory
for analytic functions, e.g., Rudin \cite[Theorem 10.22]{R74}. The proof of
Lemma \ref{lem-PL} can be found in Stein and Shakarchi \cite{SS03}.
%%%%%%%%%%%%%%%%%%%%%%%%%%%%%%%%%%%%%%%%

\section{Proof of Theorem \ref{thm-Cauchy}}\label{sec-proof}

Since $u(x_0,t) = (J^\al u)_x(x_0,t) = 0$ for $0<t<T$, we see
\begin{equation}\label{eq-3.1}
J^{m\al}u(x_0,t) = J^{m\al}u_x(x_0,t) = 0,
\quad 0<t<T,
\end{equation}
where $(m-1)\al>\f52$. By $0 < x_0 < 1$, we can choose
$\de \in (0,x_0)$ arbitrarily. Then it suffices to prove $u(x,t) = 0$ for
$\de <x<x_0$ and $0<t<T$.
Indeed, since $\de > 0$ can be arbitrarily small, we see that
$u(x,t) = 0$ for $0<x<x_0$ and $0<t<T$.
The proof for $x_0 < x < 1$ is similar.

To this end, we divide the proof into 5 steps.
%%%%%%%%%%%%%%%%%%%%

\vspace{4pt}\noindent
{\bf Step 1. }Since it was assumed $(m-1)\al > \f52$,
it follows from the Sobolev embedding that
\begin{equation}\label{eq-3.2}
H^{(m-1)\al}(0,T) \subset C^2[0,T].
\end{equation}

We can prove
\begin{equation}\label{eq-3.3}
\pa_t^\al(J^{m\al}(u-a))=(J^{m\al}u)_{x x}- p\,J^{m\al}u.
\end{equation}
The above equation must be understood in the distribution sense.
The proof of \eqref{eq-3.3} is postponed to Appendix \ref{sec-app}.

Interpreting $a$ as a constant function in $t$, we
can verify $a\in L^2(0,T;L^2(0,1))$ directly.
Hence, since $J^{m\al}u, J^{m\al}a \in\HH(0,T;L^2(0,1))$ by
$u\in L^2(0,T;H^1(0,1)) \subset L^2(0,T;L^2(0,1))$, we see that
\[
\pa_t^\al ( J^{m\al}(u-a)) = \pa_t^\al(J^{m\al}u - J^{m\al}a)
= \pa_t^\al J^{m\al}u - \pa_t^\al J^{m\al}a
= \pa_t^\al J^{m\al}u - J^{(m-1)\al}a.
\]
Therefore, setting
\begin{equation}\label{eq-3.4}
v:= J^{m\al}u \in J^{m\al}L^2(0,T;H^1(0,1)),
\end{equation}
by \eqref{eq-3.3} we obtain
\begin{equation}\label{eq-3.5}
\begin{cases}
\pa_t^\al v = v_{x x} - p\, v + J^{(m-1)\al}a
& \mbox{in }(0,1)\times(0,T),\\
v \in\HH(0,T;L^2(0,1)).
\end{cases}
\end{equation}

In addition to $\HH(0,T)$, we need spaces $H_{\ell+\si}(0,T)$ with $\ell\in\BN$
and $\si\in(0,1)$, which is defined by
$$
H_{\ell+\si}(0,T):=\left\{w\in H^{\ell+\si}(0,T);\,
w(0)=\f{\rd w}{\rd t}(0)=\cdots=\f{\rd^{\ell-1}w}{\rd t^{\ell-1}}(0)=0\right\}.
$$
Then we can readily verify that $H_{\ell+\si}(0,T) = J^{\ell+\si}L^2(0,T)$
by $H_\si (0,T) = J^\si L^2(0,T)$ for $0<\si<1$.

By \eqref{eq-3.4}, we see
$$
v \in H_{m\al}(0,T;H^1(0,1)), \quad
\paa v \in H_{(m-1)\al}(0,T;H^1(0,1)).
$$
Hence, \eqref{eq-3.5} yields
\begin{equation}\label{eq-3.6}
v_{x x} - p(x)v = \paa v - J^{(m-1)\al}a
\in H_{(m-1)\al}(0,T;L^2(0,1)).
\end{equation}
With \eqref{eq-3.6}, noting $(\de,x_0) \subset (0,1)$ and
using $v \in H_{(m-1)\al}(0,T;H^1(0,1))$,
we apply the interior regularity for an elliptic operator
$\f{\rd^2}{\rd x^2} - p(x)$ (e.g., Gilbarg and Trudinger \cite[Theorem 8.8]{GT}), so that
\begin{equation}\label{eq-3.7}
v \in H_{(m-1)\al}(0,T;H^2(\de,x_0)).
\end{equation}
Applying the trace theorem to \eqref{eq-3.7}, by \eqref{eq-3.2} we obtain
$$
g:= v_x(\de,\cdot\,) \in H_{(m-1)\al}(0,T) \subset C^2[0,T].
$$
Consequently, we obtain $g \in H_{(m-1)\al}(0,T)
= \ov{\BCC[0,T]}^{H^{(m-1)\al}(0,T)} \subset C^2[0,T]$ and
$g(0) = 0$. Similarly, in terms of \eqref{eq-3.7}, we can see $v(\,\cdot\,, 0) = 0$ in $(\de,x_0)$.
Therefore, \eqref{eq-3.2} and $H^2(\de,x_0) \subset C^1[\de,x_0]$ yield
\[
\begin{cases}
g \in {\BCC}[0,T] \cap C^2[0,T], \quad v(\,\cdot\,, 0) = 0
\mbox{ in }(\de,x_0),\\
v \in C^2([0,T];H^2(\de,x_0)) \subset C^1([0,T];C^1[\de,x_0]).
\end{cases}
\]
Moreover, by \eqref{eq-3.1} and \eqref{eq-3.7}, we have
$v(x_0,t) = v_x(x_0,t) = 0$ for $0<t<T$.

Now, in place of $u$, we consider the solution $v$ to
\begin{equation}\label{eq-3.9}
\begin{cases}
\pa_t^\al v = v_{x x} - p\, v + J^{(m-1)\al}a
& \mbox{in }(\de,x_0)\times(0,T),\\
v_x(\de,t) = g(t) \in {\BCC}[0,T] \cap C^2[0,T], \ v_x(x_0,t) = 0,
& 0<t<T,\\
v \in C^1([0,T]; C^1[\de,x_0])
\end{cases}
\end{equation}
satisfying
\begin{equation}\label{3.9}
v(x_0,t) = 0, \quad 0<t<T.
\end{equation}

We construct the following extension $G\in C^2[0,+\infty)$ of
$g \in C^2[0,T]$. We can find $G_0 \in C^2[0,+\infty)$ such that
$G_0|_{(0,T)} = g$ and $\| G_0\|_{C^2[0,T+1]}
\le C\| g\|_{C^2[0,T]}$. Let $\chi \in C^\infty[0,+\infty)$ satisfy
$\chi(t) = \begin{cases}
1, & t \le T, \\
0, & t\ge T+1.
\end{cases}$
We set
$$
G(t) = \chi(t) G_0(t), \quad t>0.
$$
Then there holds
\begin{equation}\label{eq-3.11}
\begin{gathered}
G \in C^2[0,+\infty), \quad G|_{(0,T)} = g, \quad
G|_{(T+1,+\infty)} = 0, \\
\| G\|_{C^2[0,T+1]} \le C\| g\|_{C^2[0,T]}.
\end{gathered}
\end{equation}

Now we mainly consider an initial-boundary value problem:
\begin{equation}\label{eq-3.12}
\begin{cases}
\pa_t^\al V = V_{x x} - p\, V + J^{(m-1)\al}a, & \de < x < x_0,\ t>0,\\
V_x(\de,t) = G(t), \quad V_x(x_0,t) = 0, & t>0.
\end{cases}
\end{equation}
Then by \eqref{eq-3.9}, the uniqueness of solution to the initial-boundary
value problem yields $V(x,t) = v(x,t)$ for $\de<x<x_0$ and $0<t<T$.
Taking into consideration $v \in C([\de,x_0] \times [0,T])$, by \eqref{3.9}
we derive
\begin{equation}\label{eq-3.13}
V(x_0,t) = v(x_0,t) = 0, \quad 0<t<T.\medskip
\end{equation}
%%%%%%%%%%%%%%%%%%%%

\noindent
{\bf Step 2.} In this step, we estimate $V$.
Henceforth, $\|\cdot\|$ and $(\,\cdot\,,\cdot\,)$ denote the norm
and the scalar product in $L^2(\de,x_0)$ respectively if not
specified otherwise.

Together with the existence of a solution $V$ to \eqref{eq-3.12}, we will estimate
$\|V_{x x}(\,\cdot\,,t)\|$ and $\| \pa_t^\al V(\,\cdot\,,t)\|$.
Let $\{ \la_n, \vp_n\}_{n\in \BN}$ be the eigensystem of
$A_\de v = -v_{x x} + p(x)v$ with the domain $\cD(A_\de)
= \{ \eta\in H^2(\de,x_0);\, \eta_x(\de) = \eta_x(x_0) = 0\}$.
Here we note that there exists some $n_0 \in \BN$ such that
$\la_1 < \cdots < \la_{n_0-1} < 0 \le \la_{n_0} < \la_{n_0+1}
< \cdots \to \infty$.
We define two operators with the domain $L^2(\de,x_0)$ and the range in
itself by
\begin{equation}\label{eq-def-SK}
\begin{aligned}
S(t)a & := \sum_{n=1}^\infty (a,\vp_n)E_{\al,1}(-\la_n t^\al)\vp_n(x),\\
K(t)a & := t^{\al-1}\sum_{n=1}^\infty E_{\al,\al}(-\la_n t^\al)
(a,\vp_n)\vp_n(x),
\end{aligned}\quad a \in L^2(\de,x_0),\ t>0.
\end{equation}
We have the following properties concerning $S(t)$ and $K(t)$.

\begin{lem}\label{lem-SK}
Let $S(t)$ and $K(t)$ be defined in \eqref{eq-def-SK}. Then for
$a\in L^2(\de,x_0)$ and $t>0,$ there hold
\begin{gather}
\|S(t)a\|\le C\,\e^{C t}\|a\|,\quad\lim_{t\to0+}\| S(t)a-a\|=0,\label{eq-3.15}\\
A_\de K(t)a=-S'(t)a.\label{eq-3.14}
\end{gather}
\end{lem}

Next, we set
\begin{align*}
W(x,t) & := V(x,t) + \f{(x-x_0)^2}{2(x_0-\de)}G(t),\\
F(x,t) & := \f{(x-x_0)^2}{2(x_0-\de)}(\pa_t^\al G(t) +p(x)G(t))
- \f1{x_0-\de}G(t) + J^{(m-1)\al}a\\
& =: f(x,t) + J^{(m-1)\al}a.
\end{align*}
Since $\f{(x-x_0)^2}{2(x_0-\de)}G(t) \in\HH(0,t_0;L^2(\de,x_0))$
with each $t_0>0$, we see that
$W\in\HH(\de,t_0;L^2(\de,x_0))$
if and only if $V\in\HH(0,t_0;L^2(\de,x_0))$.
Therefore, it is sufficient to consider an initial-boundary value problem
for $W$:
\begin{equation}\label{eq-3.16}
\begin{cases}
\pa_t^\al W = W_{x x} - p\, W + F(x,t), & \de < x < x_0,\ 0<t<t_0,\\
W_x(\de,t) = W_x(x_0,t) = 0, & 0<t<t_0,\\
W\in\HH(0,t_0;L^2(\de,x_0))\cap L^2(0,t_0;H^2(\de,x_0))
\end{cases}
\end{equation}
for arbitrary $t_0>0$.

\begin{lem}\label{lem-W}
Let $W$ satisfy \eqref{eq-3.16}. Then $W$ admits the representation
\begin{equation}\label{eq-3.17}
W(x,t) = \int^t_0 K(t-\tau)f(\tau)\,\rd\tau
+ \int^t_0 K(t-\tau)J^{(m-1)\al}a\,\rd\tau=: W_1(t) + W_2(t)
\end{equation}
for $t>0,$ where $K(t)$ was defined in \eqref{eq-def-SK}.
\end{lem}

For consistency, we postpone the proofs of Lemmata \ref{lem-SK}
and \ref{lem-W} to Appendix \ref{sec-app}.

Henceforth, $C>0$ and $C_k>0$ denote generic constants independent
of $t$ but may depend on $g$ and $p$.

Then we have
$$
\f{\pa f}{\pa t}(x,t)
= \f{(x-x_0)^2}{2(x_0-\de)}\left(
\f\rd{\rd t}\paa G(t) + p(x)\f{\rd G}{\rd t}(t)\right)
- \f1{x_0-\de}\f{\rd G}{\rd t}(t),
$$
where
\begin{align*}
\f\rd{\rd t}\paa G(t) &
= \f1{\Ga(1-\al)}\f\rd{\rd t}\int^t_0 \f{G'(t-s)}{s^\al}\, \rd s\\
& = \f{G'(0)}{\Ga(1-\al)}t^{-\al}
+ \f1{\Ga(1-\al)}\int^t_0 \f{G''(t-s)}{s^\al}\, \rd s
\end{align*}
by $G \in C^2[0,+\infty)$.  Hence, in terms of \eqref{eq-3.11}, we deduce
\begin{equation}\label{3.18}
\begin{aligned}
\left|\f{\pa f}{\pa t}(x,t)\right| & \le\f{C|G'(0)|}{t^\al}
+C\int^t_0\f{|G''(t-s)|}{s^\al}\,\rd s+C|G'(t)|\\
& \le C\|G\|_{C^2[0,+\infty)}(1+t^{-\al}+t^{1-\al})\\
& \le C\|g\|_{C^2[0,T]}(1+t^{-\al}+t^{1-\al}),\quad\forall\,x\in(\de,x_0),\ \forall\,t>0.
\end{aligned}
\end{equation}

Hence, by \eqref{eq-3.15}, \eqref{eq-3.14} and integration by parts, we have
\begin{align*}
A_\de W_1(t) & = \int^t_0 A_\de K(t-\tau)f(\tau)\,\rd\tau
= \int^t_0 A_\de K(\tau) f(t-\tau)\,\rd\tau\\
& = - \int^t_0 S'(\tau)f(t-\tau)\,\rd\tau
=\Big[S(\tau)f(t-\tau)\Big]^{\tau=0}_{\tau=t} - \int^t_0 S(\tau)f'(t-\tau)\,
\rd\tau\\
& = f(t) - S(t)f(0) - \int^t_0 S(\tau)f'(t-\tau)\,\rd\tau.
\end{align*}
Using \eqref{3.18}, we obtain
\begin{align*}
\| A_\de W_1(t)\| & \le C\,\e^{C t}\| f\|_{C^1([0,T];L^2(\de,x_0))}\\
& \quad\,+ \int^t_0 C\,\e^{C\tau}\|g\|_{C^2[0,T]}(1+(t-\tau)^{-\al}
+ (t-\tau)^{1-\al})\,\rd\tau,
\end{align*}
indicating
\begin{equation}\label{eq-3.18}
\| A_\de W_1(t)\| \le C(1+t+t^{1-\al}+t^{2-\al})\,\e^{C t}
\le C\,\e^{C_1t}, \quad t>0.
\end{equation}
On the other hand, since
$J^{(m-1)\al}a = \f1{\Ga((m-1)\al+1)}t^{(m-1)\al}a$, we have
\begin{align*}
\| A_\de W_2(t)\| & = \left\| -\int^t_0 S'(t-\tau) J^{(m-1)\al}a \,\rd\tau\right\|\\
& = \left\| \int^t_0 S'(t-\tau)\f1{\Ga((m-1)\al+1)}s^{(m-1)\al}a \,\rd\tau \right\|.
\end{align*}
Then, by noting $(m-1)\al > 0$, the integration by parts yields
\begin{align*}
\int^t_0 S'(t-\tau)\tau^{(m-1)\al}a \,\rd\tau
& = \Big[S(t-\tau)\tau^{(m-1)\al}\Big]^{\tau=0}_{\tau=t}\,a\\
& \quad\,+ (m-1)\al \int^t_0 S(t-\tau)\tau^{(m-1)\al-1}a \,\rd\tau\\
& = -t^{(m-1)\al} + (m-1)\al \int^t_0 S(t-\tau)\tau^{(m-1)\al-1}a\,\rd\tau
\end{align*}
and so \eqref{eq-3.15} yields
\begin{align*}
\left\| \int^t_0 S'(t-\tau)\tau^{(m-1)\al}a \,\rd\tau\right\|
& \le C\,t^{(m-1)\al} + C\int^t_0 \tau^{(m-1)\al-1} \e^{C(t-\tau)}\,\rd\tau\\
& \le C\,t^{(m-1)\al} + C\,\e^{C t}\int_{\BR_+} \tau^{(m-1)\al-1}\e^{-C\tau}\,
\rd\tau\\
& \le C\,t^{(m-1)\al} + C\,\e^{C t}\f{\Ga((m-1)\al)}{C^{(m-1)\al}}
\le C_2\,\e^{C_2t}, \quad t>0.
\end{align*}
Consequently, \eqref{eq-3.17} and \eqref{eq-3.18} imply
$$
\| W_{x x}(\,\cdot\,,t)\| \le C\| A_\de W(\,\cdot\,,t)\|
\le C_2\,\e^{C_2t}, \quad t>0,
$$
that is,
$$
\| V_{x x}(\,\cdot\,,t)\| + \| V(\,\cdot\,,t)\|
\le C_2\,\e^{C_2t}, \quad t>0.
$$
Moreover, by the Sobolev embedding $H^2(\de,x_0) \subset C[\de,x_0]$,
we have
$\| V(\,\cdot\,,t)\|_{C[\de,x_0]}\le C\,\e^{C_2t}$ for $t>0$.
Using the first equation in \eqref{eq-3.12}, we can estimate $\pa_t^\al V$ as
$$
\| \pa_t^\al V(\,\cdot\,,t)\| \le C_2\,\e^{C_2t}, \quad t>0.
$$
Therefore,
$$
\| \pa_t^\al V(\,\cdot\,,t)\|_{L^1(\de,x_0)}
+ \| V_{x x}(\,\cdot\,,t)\|_{L^1(\de,x_0)}
+ \| V(\,\cdot\,,t)\|_{C[\de,x_0]} \le C_3\,\e^{C_3t}, \quad t>0.
$$
Hence, there exists some constant $s_0>0$ such that
for any $s>s_0$, we have
$$
\pa_t^\al V(x,t)\,\e^{-s t},\ V_{x x}(x,t)\,\e^{-s t},
\ V(x,t)\,\e^{-s t} \in L^1((\de,x_0)\times\BR_+)
$$
and
\begin{equation}\label{eq-3.19}
\| V(\,\cdot\,,t)\|_{C[\de,x_0]} \le C_3\,\e^{C_3t}, \quad t>0.
\end{equation}
Hence, Fubini's theorem yields that
$$
| \pa_t^\al V(x,t)\,\e^{-s t}|,\ | V_{x x}(x,t)\,\e^{-s t}|,
\ | V(x,t)\,\e^{-s t}| \quad \mbox{are integrable in }t\in\BR_+
$$
for arbitrarily chosen $s>s_0$ and almost all $x\in (\de,x_0)$. Thus
$$
\int^\infty_0 | \pa_t^\al V(x,t)| \,\e^{-s t}\,\rd t, \quad
\int^\infty_0 | V_{x x}(x,t)| \,\e^{-s t}\,\rd t, \quad
\int^\infty_0 | V(x,t)| \,\e^{-s t}\,\rd t
$$
exist for almost all $x\in (0,x_0)$ and $s>s_0$.  This is the same for
$V_x(\de,t)$ and $V_x(x_0,t)$.

Kubica, Ryzsewska and Yamamoto \cite[Theorem 2.7]{KRY21} implies
$\wh{\pa_t^\al V}(x,s) = s^\al\wh V(x,s)$ for $s>s_0$.
Therefore, we obtain
\begin{equation}\label{eq-3.20}
\begin{cases}
s^\al\wh V(x,s) = \wh V_{x x}(x,s) - p(x)\wh V(x,s)
+ s^{-(m-1)\al-1}a, & x\in (\de,x_0),\\
\wh V_x(\de,s) = \wh G(s), \quad \wh V_x(x_0,s) = 0
\end{cases}
\end{equation}
for $s>s_0$.
%%%%%%%%%%%%%%%%%%%%

\vspace{4pt}\noindent
{\bf Step 3.} Recalling the function $\vp(x,z)$ defined by \eqref{eq-y}, we set
\begin{equation}\label{eq-3.21}
F(z):= \int^{x_0}_\de  a(x)\vp(x,z)\,\rd x, \quad z\in \BC.
\end{equation}
We notice that $F(z)$ is an entire function on $\BC$ since $\vp$ is anlytic with respect to $z\in\BC$.

\begin{lem}
The function $F(z)$ defined in \eqref{eq-3.21} satisfies
\begin{equation}\label{eq-3.22}
s^{-(m-1)\al-1}F(z)
= (s^\al+z^2)\int^{x_0}_\de \wh V(x,s)\vp(x,z)\,\rd x
+ \wh G(s) + \wh V(x_0,s)\vp'(x_0,z)
\end{equation}
for $s>s_0$ and $z\in \BC$.
\end{lem}

\begin{proof}
By \eqref{eq-3.20}, we have
$$
F(z)= s^{(m-1)\al+1}\int^{x_0}_\de \left(s^\al\wh V(x,s)
- \wh V_{x x}(x,s) + p(x)\wh V(x,s)\right) \vp(x,z)\,\rd x.
$$
Then we apply the integration by parts to obtain
\begin{align*}
\int^{x_0}_\de  \wh V_{x x}(x,s)\vp(x,z)\,\rd x
& =\left[\wh V_{x}(x,s)\vp(x,z)\right]^{x=x_0}_{x=\de}
- \left[\wh V(x,s)\vp'(x,z)\right]^{x=x_0}_{x=\de}\\
& \quad\,+ \int^{x_0}_\de \wh V(x,s)\vp_{x x}(x,z)\,\rd x\\
& = -\wh G(s) - \wh V(x_0,s)\vp'(x_0,z)
+ \int^{x_0}_\de  p(x) \vp(x,z) \wh V(x,s)\,\rd x\\
& \quad\,- z^2\int^{x_0}_\de  \wh V(x,s)\vp(x,z)\,\rd x.
\end{align*}
Hence, we obtain
$$
F(z) = s^{(m-1)\al+1}\left(
\int^{x_0}_\de  (s^\al+z^2)\wh V(x,s)\vp(x,z)\,\rd x
+ \wh G(s) + \wh V(x_0,s)\vp'(x_0,z) \right).
$$
Thus the proof of Lemma 3.3 is complete.
\end{proof}

The choice $z=\pm \ri\,s^{\f\al2}$ yields $s^\al+z^2=0$ in \eqref{eq-3.22}
and thus
\begin{equation}\label{eq-3.23}
s^{-(m-1)\al-1}F(z)=\vp'(x_0,z)\wh V(x_0,s)
+ \wh G(s), \quad z=\pm \ri\,s^{\f\al2},\ s>s_0.\medskip
\end{equation}
%%%%%%%%%%%%%%%%%%%%

\noindent
{\bf Step 4.} Based on \eqref{eq-3.23}, we can derive an estimate for
$F(z)$ defined in \eqref{eq-3.21} as follows.

\begin{lem}\label{lem-est-F}
There exists a sufficiently large integer $N$ such that
\[
|F(z)|\le C|z|^N,\quad z=\pm\ri\,s^{\f\al2},\ s > s_0,
\]
where the constant $C>0$ is independent of all $s>s_0$ and $N$.
\end{lem}

\begin{proof}
For $z=\pm\ri\,s^{\f\al2}$ for $s>s_0$, we have
$|\vp'(x_0,z)|\le C|z|\,\e^{|z|}$
by \eqref{esti-phi'} in Lemma \ref{lem-asymp}. Substituting this into
\eqref{eq-3.23} implies
$$
|F(z)|\le C\, s^{(m-1)\al+1}|z|\,\e^{|z|}\left|\int_{\BR_+}V(x_0,t)\,\e^{-s t}
\,\rd t\right|
+ s^{(m-1)\al+1}\left|\int_{\BR_+}G(t)\,\e^{-s t}\,\rd t\right|
$$
for $z=\pm \ri\,s^{\f\al2}$.
From \eqref{eq-3.13} and \eqref{eq-3.19}, it follows that
\begin{align*}
\left| \int_{\BR_+} V(x_0,t)\, \e^{-st} dt\right|
& \le \int_T^{+\infty}|V(x_0,t)|\,\e^{-s t}\,\rd t
\le \int_T^{+\infty}C_3\,\e^{(C_3-s)t}\,\rd t\\
& =\f{C_3\,\e^{(C_3-s)T}}{s-C_3},\quad s>C_3.
\end{align*}
On the other hand, by \eqref{eq-3.11}, we see that
\[
\left|\int_{\BR_+}G(t)\,\e^{-s t}\,\rd t\right|
\le \int_0^{T+1}|G(t)|\,\e^{-s t}\,\rd t
\le\|G\|_{L^1(0,T+1)}, \quad s>0.
\]
Collecting the above estimates, we reach
\[
|F(z)|\le C_4 s^{(m-\f12)\al+1}
\f{\exp(s^{\f\al2})\e^{-(s-C_5)T}}{s-C_5}
+ C_4s^{(m-1)\al+1},\quad s>C_5,
\]
where $C_5>0$ is some constant.

Since $\al\in(0,1)$, we can dominate
$|F(z)|\le C_6 s^{(m-\f12)\al+1}$ for $s > C_5$.
By the relation $z=\pm\ri\,s^{\f\al2}$, we can further conclude
\begin{equation}\label{eq-3.25}
|F(z)|\le C_7|z|^{2m-1+\f2\al}\le C_8|z|^N,
\quad z=\pm\ri\,s^{\f\al2},\ s>C_5
\end{equation}
with an integer $N\ge 2m-1+\f2\al$, which is the desired estimate.
\end{proof}

Next, we show two further properties concerning $F(z)$ below.

\begin{lem}\label{lem-poly}
The function $F(z)$ defined in \eqref{eq-3.21} is a polynomial of order $N$ at most$,$
where the integer $N$ is given in \eqref{eq-3.25}.
\end{lem}

\begin{proof}
Introducing $F_N(z):=\f{F(z)}{(z+1)^N}$, we see that $F_N(z)$ is holomorphic
for $\rRe\,z>-1$.
Furthermore, \eqref{eq-3.25} implies that there exists a sufficiently large
constant $M>0$ such that
\[
|F_N(z)|\le\f{C_8|z|^N}{|z+1|^N}\le C_9,\quad \arg z=\pm \f\pi2,\ |z|\ge M.
\]
Here and henceforth the constant $C_k$ can depend also on the
constant $M>0$. Meanwhile, for $\arg z=\pm\f\pi2$, we have
$|z+1|\ge1$ and  the continuity of the function $\varphi(x,z)$ yields
$$
|F_N(z)| = \f1{| z+1|^N}\left| \int^{x_0}_\de
a(x)\vp(x,z) \rd x \right|
\le \int_\de^{x_0} |a(x)\vp(x,z)|\,\rd x\le C_{10}
$$
if $| z| \le M$ and $\arg z = \pm \f\pi2$.

Combining the above estimates for $F_N$ yields
\[
|F_N(z)|\le C_{11} \quad \mbox{for all
$z\in \BC$ satisfying }\arg z=\pm\f\pi2.
\]

On the other hand, it follows from the asymptotic estimate \eqref{esti-asymp0}
of $\vp(x,z)$ in Lemma \ref{lem-asymp} that $|\vp(x,z)|\le C_{12}$ for any
$z\in\BR$. Hence, we obtain
\[
|F(z)|\le C_{13}\int_\de^{x_0} |a(x)|\,\rd x\le C,\quad z\in\BR
\]
and finally $|F_N(z)|\le C$ for $z \in \BR$.
Again by \eqref{esti-asymp0}, we see
\begin{align*}
| F_N(z)| \le C_{14}| F(z)|
& \le C_{14}\left| \int^{x_0}_\de  a(x)\vp(x,z) \,\rd x\right| \\
& \le C_{15}\int^{x_0}_\de  \e^{|\rIm\,z| x} | a(x)| \,
\rd x \le C_{15}\,\e^{|z| x_0}
\end{align*}
for $\rRe\,z>M$. In the case of $0<\rRe\,z\le M$, we can conclude from the continuity of the function $\vp(x,z)$ that
$$
|F_N(z)|\le C_{14}\|a\|_{L^\infty(0,1)}\|\vp\|_{L^\infty((0,1)\times\{|z|\le M\}}.
$$
Combining the above estimates, we finally get
$$
|F_N(z)|\le C_{15}'\,\e^{|z|x_0},\quad\rRe z>0.
$$

Choosing $(\te_1,\te_2) = (0,\f\pi2)$ and $(\te_1,\te_2) = (-\f\pi2,0)$, we apply Lemma \ref{lem-PL} to obtain $|F_N(z)|\le C$ for all $\rRe\,z>0$. Hence,
\[
|F(z)|\le C|z+1|^N ,\quad\rRe\,z\ge0.
\]
Similarly, by considering $\f{F(z)}{(z-1)^N}$, we can derive $|F(z)|
\le C|z-1|^N$
for $\rRe\,z<0$. Consequently, according to Lemma \ref{lem-GL}, $F$ must be a
polynomial satisfying $\deg F\le N$.
\end{proof}

\begin{lem}\label{lem-vanish}
The function $F(z)$ defined in \eqref{eq-3.21} vanishes identically in $\BC$.
\end{lem}

\begin{proof}
Owing to Lemma \ref{lem-poly}, we can assume that
$F(z)=\sum^N_{j=0}a_jz^j$ with some $a_j \in \BC$.
For $z>0$, let us consider $\lim_{z\to+\infty}F(z)$. From the asymptotic
behavior \eqref{esti-asymp} of $\vp(x,z)$ in Lemma \ref{lem-asymp}, we obtain
\begin{align*}
\lim_{z\to+\infty}F(z) & =\lim_{z\to+\infty}\left(\int_\de^{x_0}
a(x)\cos(z(x-\de))\,\rd x
+ \int_\de^{x_0} a(x)O(|z|^{-1})\,\rd x\right)\\
& =\lim_{z\to+\infty}\int_\de^{x_0} a(x)\cos(z (x-\de))\,\rd x.
\end{align*}
In view of the Riemann-Lebesgue lemma, we have
\[
\lim_{z\to+\infty}\int_\de^{x_0} a(x)\cos(z (x-\de))\,\rd x=0,
\]
that is,
\[
\lim_{z\to+\infty}F(z)=\lim_{z\to+\infty}\sum^N_{j=0}a_j z^j=0.
\]
If $a_N\ne 0$, then
\[
\lim_{z\to+\infty}F(z)=\lim_{z\to+\infty}a_n z^N
\left(1+\f{a_{N-1}}{a_N}\f1z+\cdots+\f{a_0}{a_N}\f1{z^N}\right)=\infty,
\]
which is impossible. Therefore, we conclude $a_N=0$. Repeating the same argument,
we see $a_j=0$ for $j=0,1,\ldots,N$, which completes the proof.
\end{proof}
%%%%%%%%%%%%%%%%%%%%

\noindent
{\bf Step 5.} Now we are ready to finish the proof of Theorem \ref{thm-Cauchy}.
By \eqref{eq-3.21} and Lemma \ref{lem-vanish}, we see that
\begin{equation}\label{eq-3.28}
\int_\de^{x_0} a(x)\vp(x,z)\,\rd x=0,\quad\forall\,z\in\BC.
\end{equation}
In order to show $a \equiv0$ in $(\de,x_0)$, we invoke the
Neumann eigensystem
$\{(\mu_n,\psi_n)\}$ of the operator $-\f{\rd^2}{\rd x^2}+p(x)+p_0$ in $(\de,x_0)$, that is,
$$
\begin{cases}
-\psi_n''+ (p(x)+p_0)\psi_n = \mu_n\psi_n & \mbox{in }(\de,x_0),\\
\psi_n'(\de)=\psi_n'(x_0)=0.
\end{cases}
$$
Here $p_0>0$ is a constant sufficiently large so that $\mu_n>0$ for $n\in \BN$.
Normalizing $\psi_n$ by $\psi_n(\de)=1$, we immediately see that
$\vp(\,\cdot\,, \sqrt{\mu_n-p_0}\,) = \psi_n$. Since $\{\psi_n\}$ forms a complete
orthogonal basis in $L^2(\de,x_0)$, it suffices to take
$z=\sqrt{\mu_n-p_0}$ in \eqref{eq-3.28} to conclude $a\equiv0$ in $(\de,x_0)$.

It remains to show $G\equiv0$. Now the equation \eqref{eq-3.23} becomes
\[
\int_{\BR_+}G(t)\,\e^{-s t}\,\rd t
+ \vp'(x_0,z)\int_{\BR_+}V(x_0,t)\,\e^{-s t}\,\rd t=0,
\quad z=\pm\ri\,s^{\f\al2}.
\]
Using \eqref{eq-3.11}, \eqref{eq-3.13}, \eqref{eq-3.23} and
Lemma \ref{lem-vanish}, for $z= \pm \ri\, s^{\f\al2}$
we deduce
\begin{align*}
\int^{T+1}_0 G(t)\,\e^{-s t} \,\rd t
& = \int_{\BR_+} G(t)\,\e^{-s t} \,\rd t
= -\vp'(x_0,z)\int_{\BR_+} V(x_0,t)\,\e^{-s t} \,\rd t\\
& = -\vp'(x_0,z)\int^{+\infty}_T V(x_0,t)\,\e^{-s t} \,\rd t.
\end{align*}
Therefore, \eqref{esti-phi'} in Lemma \ref{lem-asymp} implies
\begin{align*}
\left|\int^T_0 g(t)\,\e^{-s t}\,\rd t\right|
& =\left|\int^T_0G(t)\,\e^{-s t}\,\rd t\right|
=\left|\int^{T+1}_0G(t)\,\e^{-s t}\,\rd t
-\int^{T+1}_T G(t)\,\e^{-s t}\,\rd t\right|\\
& \le\left|\int^{T+1}_0G(t)\,\e^{-s t}\,\rd t\right|
+\int^{T+1}_T|G(t)|\,\e^{-s t}\,\rd t\\
& \le|\vp'(x_0,z)|\left|\int_T^{+\infty}V(x_0,t)\,\e^{-s t}\,\rd t\right|
+\|G\|_{C[0,T+1]}\int^{T+1}_T\e^{-s t}\,\rd t\\
& \le\|G\|_{C[0,T+1]}\,\e^{-s T}\f{1-\e^{-s}}s\\
& \quad\,+C s^{\f\al2}\exp(C s^{\f\al2})\left|\int_T^\infty V(x_0,t)\,
\e^{-s t}\,\rd t\right|,\quad s>\eta_0.
\end{align*}
Moreover, by \eqref{eq-3.19}, we obtain
\[
\left|\int_0^T g(t)\,\e^{-s t}\,\rd t\right|\le C s^{-1}\e^{-s T}
+C s^{\f\al2}\exp(C s^{\f\al2})\f{\e^{-(s-C_3')T}}{s-C_3'}
\]
for $s>C_3':=\max\{C_3,\eta_0\}$. For any $\ve>0$, we see that
\begin{align*}
\left|\int_0^{T-\ve}g(t)\,\e^{-s t}\,\rd t\right|
& \le C s^{-1}\e^{-s T}+C s^{\f\al2} \exp(C s^{\f\al2})\f{\e^{-(s-C_3')T}}{s-C_3'}
+\int_{T-\ve}^T|g(t)|\,\e^{-s t}\,\rd t\\
& \le C s^{-1}\e^{-s T}
+C s^{\f\al2}\exp(C s^{\f\al2}-\ve s)\f{\e^{C_3'T}}{s-C_3'}\e^{-s(T-\ve)}\\
& \quad\,+\f{C|\e^{-s(T-\ve)}-\e^{-s T}|}s, \quad s>C_3'.
\end{align*}
Choosing $s>0$ large enough, we arrive at
\[
\left|\int_0^{T-\ve}g(t)\,\e^{s(T-\ve-t)}\,\rd t\right|\le C,\quad s > C_3'.
\]
Introducing $\wt g(t):=g(T-\ve-t)$, we immediately have
\[
\left|\int_0^{T-\ve}\wt g(t)\,\e^{s t}\,\rd t\right|\le C,\quad s>C_3'.
\]
For $0<s\le C_3'$, we estimate
\[
\left|\int_0^{T-\ve}\wt g(t)\,\e^{s t}\,\rd t\right|
\le\e^{C_3'(T-\ve)}\int_0^{T-\ve}|\wt g(t)|\,\rd t
=\e^{C_3'(T-\ve)}\|g\|_{L^1(0,T-\ve)}.
\]
Therefore, by defining
$\wt G(z):=\int_0^{T-\ve}\wt g(t)\,\e^{z t}\,\rd t$, we have
\[
\left|\wt G(z)\right|\le C_{16}
:=\max\left\{C,\e^{C_3'(T-\ve)}\|g\|_{L^1(0,T-\ve)}\right\},\quad\arg z=0.
\]
Meanwhile, for $\arg z=\f\pi2$, it is readily seen that
\[
\left|\wt G(z)\right|\le\int_0^{T-\ve}|\wt g(t)|\,\rd t
=\|g\|_{L^1(0,T-\ve)}.
\]
For $0<\arg z<\f\pi2$, we estimate
\[
\left|\wt G(z)\right|\le\int_0^{T-\ve}|\wt g(t)|\,\e^{|z|t}\,\rd t
\le\|g\|_{L^1(0,T-\ve)}\,\e^{(T-\ve)|z|}.
\]
Then we can apply Lemma \ref{lem-PL} with
$\te_1=0$, $\te_2=\f\pi2$ and $\ga=1$ to obtain
\[
\left|\wt G(z)\right|\le\max\left\{C_{16},\|g\|_{L^1(0,T-\ve)}\right\},
\quad0\le\arg z\le\f\pi2.
\]
Similarly, $\wt G(z)$ is also bounded for $-\f\pi2\le\arg z\le0$.
On the other hand, it follows immediately from
the definition of $\wt G(z)$ that it is bounded for $\rRe\,z<0$.
Now that $\wt G(z)$ is bounded and holomorphic in the whole complex plane,
Liouville's theorem guarantees that $\wt G(z)$ is a constant.
Finally, since $\lim_{z\to-\infty}\wt G(z)=0$,
we conclude $\wt G(z)\equiv0$ in $\BC$ and thus $g\equiv0$ in $(0,T-\ve)$.
Since $\ve>0$ can be arbitrarily chosen,
we obtain $g\equiv0$ in $(0,T)$ and eventually $G\equiv0$ in $\BR_+$.

Finally, by the uniqueness of the solution to
the initial-boundary value problem \eqref{eq-3.12}, we obtain
$V\equiv 0$ in $(\de,x_0) \times (0,T)$.
Hence, by \eqref{eq-3.13} we conclude that
$v\equiv 0$ in $(\de,x_0)\times (0,T)$.  Since $v=J^{m\al}u$ in
$(\de,x_0) \times (0,T)$, we immediately see that $u\equiv 0$ in
$(\de,x_0) \times (0,T)$. This completes the proof of Theorem \ref{thm-Cauchy}.
%%%%%%%%%%%%%%%%%%%%%%%%%%%%%%%%%%%%%%%%

\section{Proof of Theorem \ref{thm-isp}}\label{sec-isp}

In order to prove Theorem \ref{thm-isp}, we show several useful lemmata.
Let $\mu$ be the $(1-\al)$-th Riemann-Liouville derivative of $\rho$, i.e.,
\[
\mu(t)=(D_t^{1-\al}\rho)(t)
=\f1{\Ga(\al)}\f\rd{\rd t}\int_0^t(t-\tau)^{1-\al}\rho(\tau)\,\rd\tau.
\]

\begin{lem}\label{lem-new}
Let $h\in\HH(0,T)$ and $\rho\in H^1(0,T)$ satisfy $\rho(0)\ne0$.
Then the integral equation
\begin{equation}\label{eq-w}
h(t)=\int_0^t\mu(t-\tau)w(\tau)\,\rd\tau
\end{equation}
admits a unique solution $w\in L^2(0,T)$.
\end{lem}

\begin{proof}
Performing $J^{1-\al}$ on both sides of \eqref{eq-w} and by direct calculation,
we deduce
\[
J^{1-\al}h(t)=\int_0^t\rho(t-\tau)w(\tau)\,\rd\tau.
\]
By $h\in H_\al(0,T)$ and $\pa_t^\al=(J^\al)^{-1}$, we see
\[
\f\rd{\rd t}(J^{1-\al}h)=J^{-1}J^{1-\al}h=J^{-1}J^1\pa_t^\al h=\pa_t^\al h.
\]
Since $\rho\in H^1(0,T)\subset C[0,T]$ by the Sobolev embedding,
we differentiate the above equation to derive
\begin{equation}\label{eq-v-II}
\pa_t^\al h(t)=\rho(0)w(t) + \int_0^t\rho'(t-\tau)w(\tau)\,\rd\tau.
\end{equation}
Defining an operator $K:L^2(0,T)\longrightarrow L^2(0,T)$ by
\[(K w)(t)=\int_0^t\rho'(t-\tau)w(\tau)\,\rd\tau,\quad0<t<T,
\]
we see that \eqref{eq-v-II} can be rephrased as
\[
\pa_t^\al h(t)=\rho(0)w(t) + K w(t).
\]
By $\rho'\in L^2(0,T)$, it is not difficult to verify that
$K$ is an integral operator of the Hilbert-Schmidt type.
Then it follows immediately from \cite[Chapter X.2]{Y80} that
$K$ is a compact operator. Moreover, in view of
Gr\"onwall's inequality, we can show that $\rho(0)w + K w=0$ implies $w=0$.
Therefore, it follows from the Fredholm alternative that \eqref{eq-v-II} admits a
unique solution $w\in L^2(0,T)$.
\end{proof}

\begin{lem}[Duhamel's principle]\label{lem-Duhamel}
Let $f\in L^2(0,1),$ $\rho\in H^1(0,T),$ $\rho(0)\ne0$ and
$y\in \HH(0,T;H^1(0,1))$ satisfy \eqref{eq-1.2}. If
\begin{equation}\label{eq-convo}
y(\,\cdot\,,t)=\int_0^t\mu(t-\tau)u(\,\cdot\,,\tau)\,\rd\tau,
\end{equation}
then $u\in L^2(0,T;H^1(0,1))$
satisfies $u-f\in\HH(0,T;H^{-1}(0,1))$ and a homogeneous equation
with the initial value $f${\rm:}
\begin{equation}\label{eq-v-f}
\pa_t^\al(u-f)-u_{x x}+p(x)u=0\quad\mbox{in }(0,1)\times(0,T_*).
\end{equation}
Here $T_*>0$ is some constant.
\end{lem}

\begin{rem}
{\rm Concerning Duhamel's principle for time-fractional partial differential equations
in different function spaces, we refer e.g.\! to {\rm\cite{LRY16,JLLY}}
and the survey {\rm\cite{U19}}. In comparison with existing literature, we do not attach
the governing equations with boundary conditions in Lemma \ref{lem-Duhamel}.
Therefore, here we only focus on the representation \eqref{eq-convo} instead of
the uniqueness issue.}
\end{rem}

\begin{proof}[Proof of Lemma $\ref{lem-Duhamel}$]
Let $u$ satisfy \eqref{eq-v-f} with the regularity assumed in
Lemma \ref{lem-Duhamel}.  Denoting the right-hand side of \eqref{eq-convo}
by $\wt y$\,, i.e.,
\begin{equation}\label{eq-4.5}
\wt y(\,\cdot\,,t)=\int_0^t\mu(t-\tau)u(\,\cdot\,,\tau)\,\rd\tau,
\end{equation}
we can show that $\wt y$ belongs to $\HH(0,T;H^1(0,1))$.
Next, we know that \eqref{eq-v-f} is equivalent to
\begin{equation}\label{eq-4.6}
u=J^\al(u_{x x}-p(x)u)+f\quad\mbox{in }L^2(0,T;H^{-1}(0,1)).
\end{equation}
Similarly, \eqref{eq-1.2} is equivalent to
\begin{equation}\label{eq-4.7}
\begin{aligned}
y & =J^\al(y_{x x}-p(x)y)+(J^\al\rho)f\\
& =J^\al(y_{x x}-p(x)y)+(J^1\mu)f
\quad\mbox{in }L^2(0,T;H^{-1}(0,1)),
\end{aligned}
\end{equation}
where we used $\mu=D_t^{1-\al}\rho$.
We follow the argument used in the proof of \cite[Theorem 2.6]{JLLY} to obtain
\begin{align*}
J^\al(\wt y_{x x}-p(x)\wt y\,)+(J^1\mu)f
& =J^\al\int_0^t\mu(t-\tau)(u_{x x}-p(x)u)(\tau)\,\rd\tau
+\left(\int_0^t\mu(\tau)\,\rd\tau\right)f\\
& =\int_0^t\mu(t-\tau)\{J^\al(u_{x x}-p(x)u)(\tau)+f\}\,\rd\tau.
\end{align*}
By the Titchmarsh convolution theorem,
it turns out that if $\wt y$ satisfies \eqref{eq-4.7},
then $u$ satisfies \eqref{eq-4.6}, where $T>0$ is replaced by some $T_*>0$.
This completes the proof.
\end{proof}

Now we are well prepared to complete the proof of Theorem \ref{thm-isp}.
Assume that $y\in \HH(0,T;H^1(0,1))$
satisfies \eqref{eq-1.2} along with lateral Cauchy data
$J^\al y(x_0,t)=(J^\al y)_x(x_0,t)=0$.
Employing Lemma \ref{lem-Duhamel}, we obtain
\begin{align*}
0 & =J^\al y(x_0,t)=\int_0^t\mu(t-\tau)J^\al u(x_0,\tau)\,\rd\tau,\\
0 & =(J^\al y)_x(x_0,t)=\int_0^t\mu(t-\tau)J^\al u_x(x_0,\tau)\,\rd\tau,
\end{align*}
where $u$ satisfies \eqref{eq-v-f}.
According to Lemma \ref{lem-new}, we immediately conclude
$J^\al u(x_0,\cdot\,)$ $=J^\al u_x(x_0,\cdot\,)=0$ in $(0,T_*)$.
Finally, a direct application of
Theorem \ref{thm-Cauchy} indicates $u\equiv0$ in $(0,1)\times(0,T_*)$,
which indicates $f\equiv0$ in $(0,1)$ automatically as
the hidden initial value in \eqref{eq-1.2}.
%%%%%%%%%%%%%%%%%%%%%%%%%%%%%%%%%%%%%%%%

\section{Concluding remarks}\label{sec-rem}

In this paper, we obtained novel sharp uniqueness for an inverse
$x$-source problem for a one-dimensional time-fractional diffusion equation
with a potential. With the aid of Duhamel's principle, the key ingredient
reveals
to be the uniqueness for the lateral Cauchy problem for the corresponding
homogeneous equation. Taking Laplace transform, we changed the original
lateral Cauchy problem to an integral equation involving the initial and
boundary values of the solution. Then we managed to prove the uniqueness by
employing the Phragm\'en-Lindel\"of principle and a generalized Liouville's
theorem (Phragm\'en-Lindel\"of-Liouville argument for short). As a
byproduct, we also established a classical unique continuation property.

Let us mention that the Phragm\'en-Lindel\"of-Liouville argument used in the
proof heavily relies on the dimension in space. It is interesting to
investigate the lateral Cauchy problem for time-fractional diffusion
equations in higher spatial dimensions.
%%%%%%%%%%%%%%%%%%%%%%%%%%%%%%%%%%%%%%%%

\appendix
\section{Proofs of \eqref{1.6}, \eqref{eq-3.3} and
Lemmata \ref{lem-SK}--\ref{lem-W}}\label{sec-app}

\begin{proof}[Proof of \eqref{1.6}]
The proof is similar to the one of \eqref{eq-3.7}.
By \eqref{1.3} and $\paa (u-a) = J^{-\al}(u-a)$ for
$u-a \in\HH(0,T;H^{-1}(0,1))$, using
also $p\in L^\infty(0,1)$, we see
$$
u-a - J^\al u_{x x}(x,t) + J^\al(p u) = J^\al F,
$$
that is,
$$
(J^\al u)_{x x}(x,t) = u-a + J^\al(p u) - J^\al F
\in L^2(0,T;L^2(0,1)).
$$
Then, since $J^\al u \in L^2(0,T;H^1(0,1))$ by the second
regularity condition in \eqref{1.4},
the interior regularity for an elliptic operator
$\f{\rd^2}{\rd x^2}$ (e.g., \cite[Theorem 8.8]{GT}) yields
\eqref{1.6}, and the proof of \eqref{1.6} is complete.
\end{proof}

\begin{proof}[Proof of \eqref{eq-3.3}]
We set $H^{-2}(0,1):=(H^2_0(0,1))'$, where
$H^2_0(0,1):=\{v\in H^2(0,1)\mid v(0)=v_x(0)=v(1)=v_x(1)=0\}$.
First for $\ga>0$, we prove
\begin{equation}\label{A.1}
J^\ga\left(_{H^{-2}(0,1)}\langle w(\,\cdot\,,t),\psi\rangle_{H^2_0(0,1)}\right)
={}_{H^{-2}(0,1)}\langle J^\ga w(\,\cdot\,,t),\psi\rangle_{H^2_0(0,1)}
\end{equation}
for $w\in L^2(0,T;H^{-2}(0,1))$, $\psi\in H^2_0(0,1)$ and almost all $t\in (0,T)$.
We note that
$$
J^\ga(w(\,\cdot\,,t),\psi)_{L^2(0,1)}=(J^\ga w(\,\cdot\,,t),\psi)_{L^2(0,1)}
$$
is directly seen by Fubini's theorem, but \eqref{A.1} with
$_{H^{-2}(0,1)}\langle w(\,\cdot\,,t),\psi\rangle_{H^2_0(0,1)}$ is not trivial.

We can verify \eqref{A.1} as follows. For $w\in L^2(0,T;H^{-2}(0,1))$, by the
definition of Bochner's integral (e.g., Yosida \cite{Y80}), we can choose
a sequence $w_n(x,t):=\sum_{j=1}^{m(n)}a_j^n(x)\chi_{E_j^n}(t)$ ($n\in\BN$)
of simple functions such that $w_n\longrightarrow w$ in $L^2(0,T;$ $H^{-2}(0,1))$.
Here $m(n)\in\BN$, $a_j^n\in H^{-2}(0,1)$, $E_j^n\subset(0,T)$ are
measurable sets, and $\chi_E$ indicates the characteristic function of $E$. Then
$$
_{H^{-2}(0,1)}\langle w_n(\,\cdot\,,t),\psi\rangle_{H^2_0(0,1)}\longrightarrow
{}_{H^{-2}(0,1)}\langle w(\,\cdot\,,t),\psi\rangle_{H^2_0(0,1)}\quad\mbox{in }L^2(0,T)
$$
as $n\to\infty$ for all $\psi\in H^2_0(0,1)$ and almost all $t\in(0,T)$. Moreover,
\begin{align*}
J^\ga\left(_{H^{-2}(0,1)}\langle w_n(\,\cdot\,,t),\psi\rangle_{H^2_0(0,1)}\right)
& =J^\ga\left(\sum_{j=1}^{m(n)}
{}_{H^{-2}(0,1)}\langle a_j^n,\psi\rangle _{H^2_0(0,1)}\chi_{E_j^n}(t)\right)\\
& ={}_{H^{-2}(0,1)}\langle J^\ga w_n(\,\cdot\,,t),\psi\rangle _{H^2_0(0,1)}
\end{align*}
for all $n\in\BN$, $\psi\in H^2_0(0,1)$ and almost all $t\in(0,T)$. Using that
$J^\ga:L^2(0,T)\longrightarrow L^2(0,T)$ is continuous and letting $n\to\infty$,
we reach \eqref{A.1}.

Next, we notice that $u\in L^2(0,T;H^1(0,1))$ is a weak solution to \eqref{eq-1.4}
and satisfies $u-a\in H_\al(0,T;H^{-1}(0,1))$, so that we have the weak form:
$$
_{H^{-2}(0,1)}\langle\pa_t^\al(u-a)(\,\cdot\,,t),\psi\rangle_{H_0^2(0,1)}
+(u_x(\,\cdot\,,t),\psi_x)_{L^2(0,1)}+(p\,u(\,\cdot\,,t),\psi)_{L^2(0,1)}=0
$$
for all $\psi\in H^2_0(0,1)$ and almost all $t \in (0,T)$.
Moreover, by integration by parts, we have
$$
_{H^{-2}(0,1)}\langle\pa_t^\al(u-a)(\,\cdot\,,t),\psi\rangle_{H_0^2(0,1)}
-(u(\,\cdot\,,t),\psi_{x x})_{L^2(0,1)}+(p\,u(\,\cdot\,,t),\psi)_{L^2(0,1)}=0
$$
for all $\psi\in H^2_0(0,1)$ and almost all $t \in (0,T)$.
We operate $J^{m\al}$ to both sides to have
\begin{align*}
0 & =J^{m\al}\left(_{H^{-2}(0,1)}\langle\pa_t^\al(u-a)(\,\cdot\,,t),\psi\rangle_{H_0^2(0,1)}\right)
-(J^{m\al}u(\,\cdot\,,t),\psi_{x x})_{L^2(0,1)}\\
& \quad\,+(p\,J^{m\al}u(\,\cdot\,,t),\psi)_{L^2(0,1)}
\end{align*}
for all $\psi\in H^2_0(0,1)$ and almost all $t \in (0,T)$. Here we used
$J^{m\al}(u(\,\cdot\,,t),\psi)_{L^2(0,1)}=(J^{m\al}u(\,\cdot\,,t),\psi)_{L^2(0,1)}$ etc.,
which is immediately seen.

Now we apply \eqref{A.1} with $\ga=m\al$ to obtain
\begin{align*}
0 & ={}_{H^{-2}(0,1)}\langle J^{m\al}\pa_t^\al(u-a)(\,\cdot\,,t),\psi\rangle_{H_0^2(0,1)}
-(J^{m\al}u(\,\cdot\,,t),\psi_{x x})_{L^2(0,1)}\\
& \quad\,+(p\,J^{m\al}u(\,\cdot\,,t),\psi)_{L^2(0,1)}
\end{align*}
for all $\psi\in H^2_0(0,1)$ and almost all $t \in (0,T)$.
By $u-a\in H_\al(0,T;H^{-1}(0,1))\subset H_\al(0,T;$ $H^{-2}(0,1))$, we see
$\pa_t^\al(u-a)\in L^2(0,T;H^{-2}(0,1))$, so that
$J^{m\al}\pa_t^\al(u-a)=\pa_t^\al J^{m\al}(u-a)$. Therefore, since
$\psi\in H^2_0(0,1)$ is arbitrary, we reach
$$
\pa_t^\al J^{m\al}(u-a)(\,\cdot\,,t)-(J^{m\al}u)_{x x}(\,\cdot\,,t)
+p\,J^{m\al}u(\,\cdot\,,t)=0\quad\mbox{in }H^{-2}(0,1)
$$
for almost all $t\in(0,T)$. Thus the proof of \eqref{eq-3.3} is complete.
\end{proof}

\begin{proof}[Proof of Lemma $\ref{lem-SK}$]
First by direct calculation, it is not difficult to arrive at the following estimate
$$
\|S(t)a\|^2 \le\sum_{n=1}^\infty |(a,\vp_n)|^2|E_{\al,1}(-\la_n t^\al) |^2.
$$
Decomposing the above summation into two parts according to the signs of $\la_n$,
we further derive
$$
\|S(t)a\|^2\le C\sum_{n=1}^{n_0-1}|(a,\vp_n)|^2\exp(|\la_n|^{\f2\al}t^2)
+C\sum_{n=n_0}^\infty\f{|(a,\vp_n)|^2}{(1+|\la_n|t^\al)^2},
$$
where we applied the asymptotic estimates for the Mittag-Leffler functions
(see Podlubny \cite[Theorem 1.5]{P99} for the case of $n\le n_0$
and \cite[Theorem 1.6]{P99} for that of $n>n_0$).
Moreover, from the fact that $|\la_1|\ge|\la_n|$ ($1\le n\le n_0$),
it follows that
$$
\|S(t)a\|\le C\exp(|\la_1|^{\f1\al}t)\sum_{n=1}^\infty|(a,\vp_n)|^2
=C\exp(|\la_1|^{\f1\al}t)\|a\|.
$$
Next, we check the convergence in \eqref{eq-3.15}.
For this, it suffices to evaluate the following series
$$
\|S(t)a-a\|^2=\sum_{n=1}^\infty|(a,\vp_n)|^2|E_{\al,1}(-\la_n t^\al)-1|^2.
$$
From the above calculation, it follows that
$|E_{\al,1}(-\la_n t^\al)|\le C\exp(|\la_1|^{\f1\al}t)$. Together with the continuity of
the Mittag-Leffler function and the dominated convergence theorem, we see that
$$
\lim_{t\to0}\|S(t)a -a\|^2=\sum_{n=1}^\infty|(a,\vp_n)|^2
\lim_{t\to0}|E_{\al,1}(-\la_n t^\al) -1|^2=0.
$$
For \eqref{eq-3.14}, we employ the termwise differentiability
of the series in $S(t)a$ and the formula
\[
E_{\al,1}(-\la_n t^\al) = -\la_n t^{\al-1} E_{\al,\al}(-\la_n t^\al)
\]
to calculate the derivative $S'(t)a$ as
\begin{align*}
S'(t)a=-\sum_{n=1}^\infty t^{\al-1}\la_n(a,\vp_n)E_{\al,1}(-\la_n t^\al)\vp_n.
\end{align*}
Then the definition of $K(t)$ implies immediately $A_\de K(t)a=-S'(t)a$.
This completes the proof of Lemma \ref{lem-SK}.
\end{proof}

\begin{proof}[Proof of Lemma $\ref{lem-W}$]
Notice that the eigenfunctions $\{\vp_n\}$ forms an orthonormal basis of
$L^2(\de,x_0)$. Then by the Fourier expansion argument e.g.\! used in the proof
of Sakamoto and Yamamoto \cite[Theorem 2.1]{SY11}, we similarly obtain that
the solution $W$ to the problem \eqref{eq-3.16} admits the following series representation
\[
W(x,t)=\sum_{n=1}^\infty\left(\int_0^t\tau^{\al-1}E_{\al,\al}(-\la_n\tau^\al)
(F(\,\cdot\,,t-\tau),\vp_n)\,\rd\tau\right)\vp_n(x).
\]
Now it remains to check that the above defined Fourier series converges in the sense of
$L^2(0,t_0;L^2(\de,x_0))$. Indeed, introducing
\[
W_N(x,t):=\sum_{n=1}^N\left(\int_0^t\tau^{\al-1}E_{\al,\al}(-\la_n\tau^\al)
(F(\,\cdot\,,t-\tau),\vp_n)\,\rd\tau\right)\vp_n(x)
\]
for $N\in\BN$, it is readily seen that
\[
\|W_N(\,\cdot\,,t)\|^2\le\sum_{n=1}^N\left(\int_0^t
\tau^{\al-1}E_{\al,\al}(-\la_n\tau^\al)(F(\,\cdot\,,t-\tau),\vp_n)\,\rd\tau\right)^2.
\]
From Young's convolution inequality and using $E_{\al,\al}(-\la_n t^\al)\ge0$
(see e.g. \cite[Lemma 3.3]{SY11}), it follows that
\begin{align*}
\|W_N\|_{L^2(0,t_0;L^2(\de,x_0))}^2
& \le\sum_{n=1}^N\int_0^{t_0}\left(\int_0^t
\tau^{\al-1}E_{\al,\al}(-\la_n\tau^\al)(F(\,\cdot\,,t-\tau),\vp_n)\,\rd\tau\right)^2\rd t\\
& \le\sum_{n=1}^N\left(\int_0^{t_0}t^{\al-1}E_{\al,\al}(-\la_n t^\al)\,\rd t\right)^2
\int_0^{t_0}|(F(\,\cdot\,,t),\vp_n)|^2\rd t.
\end{align*}
Similarly to the proof of Lemma \ref{lem-SK}, we can see that
$$
|E_{\al,\al}(-\la_n t^\al)|\le C\exp(|\la_1|^{\f1\al}t),\quad\forall\,n\in\BN,
$$
which implies that
\[
\|W_N\|_{L^2(0,t_0;L^2(\de,x_0))}^2
\le C\sum_{n=1}^N\int_0^{t_0}|(F(\,\cdot\,\tau),\vp_n)|^2\rd t
\le C\|F\|_{L^2(0,t_0;L^2(\de,x_0))}^2.
\]
This implies the uniform boundedness of $\{W_N\}$ in $L^2(0,t_0;L^2(\de,x_0))$
and by passing $N\to\infty$, we obtain $W\in L^2(0,t_0;L^2(\de,x_0))$.
The proof of Lemma \ref{lem-W} is completed.
\end{proof}
%%%%%%%%%%%%%%%%%%%%%%%%%%%%%%%%%%%%%%%%

\section*{Acknowledgement}
The authors thank the anonymous referees for their valuable comments.
Z.\! Li is supported by National Natural Science Foundation of China (NSFC) (No.\! 11801326).
Y.\! Liu is supported by Grant-in-Aid for Early Career Scientists 20K14355 and 22K13954, Japan Society for the Promotion of Science (JSPS).
M.\! Yamamoto is supported by Grant-in-Aid for Scientific Research (A) 20H00117 and Grant-in-Aid for Challenging Research (Pioneering) 21K18142, JSPS and by NSFC (Nos.\! 11771270, 91730303).
%%%%%%%%%%%%%%%%%%%%%%%%%%%%%%%%%%%%%%%%

\end{document}